\documentclass[a4paper,reqno,11pt]{amsart}

\usepackage{fullpage}
\usepackage{microtype}
\usepackage[utf8]{inputenc}

\usepackage{mathtools}
\usepackage{amssymb}

\usepackage{enumitem}
\setlist[enumerate]{label=(\roman*), font=\normalfont}

\usepackage{graphicx}
\usepackage[dvipsnames]{xcolor}
\usepackage{float}
\usepackage{tabularray}

\usepackage[all]{xy}
\usepackage{tikz}
\usetikzlibrary{decorations.pathreplacing}
\usetikzlibrary{graphs,graphs.standard,calc}

\usepackage[%
    bookmarks=true,%
    bookmarksnumbered=true,%
    colorlinks=true,%
    setpagesize=false,%
    linkcolor=blue,%
    citecolor=magenta,%
    filecolor=magenta,%
    urlcolor=magenta,%
]{hyperref}

\usepackage{amsthm}
\usepackage{cleveref}

\def\cCrefname#1#2#3{%
    \crefname{#1}{#2}{#3}%
    \Crefname{#1}{#2}{#3}%
}

\def\crefparenedname#1#2#3{%
    \crefformat{#1}{#2##2##1##3#3}%
    \crefrangeformat{#1}{#2##3##1##4#3\crefrangeconjunction#2##5##2##6#3}%
    \crefmultiformat{#1}{#2##2##1##3#3}{\crefpairconjunction#2##2##1##3#3}{\crefmiddleconjunction#2##2##1##3#3}{\creflastconjunction#2##2##1##3#3}%
}
\def\Crefparenedname#1#2#3{%
    \Crefformat{#1}{#2##2##1##3#3}%
    \Crefrangeformat{#1}{#2##3##1##4#3\crefrangeconjunction#2##5##2##6#3}%
    \Crefmultiformat{#1}{#2##2##1##3#3}{\crefpairconjunction#2##2##1##3#3}{\crefmiddleconjunction#2##2##1##3#3}{\creflastconjunction#2##2##1##3#3}%
}
\def\cCrefparenedname#1#2#3{%
    \crefparenedname{#1}{#2}{#3}%
    \Crefparenedname{#1}{#2}{#3}%
}

\cCrefname{part}{Part}{Parts}
\cCrefname{chapter}{Chapter}{Chapters}
\cCrefname{section}{Section}{Sections}
\cCrefname{subsection}{Section}{Sections}
\cCrefname{subsubsection}{Section}{Sections}
\cCrefname{appendix}{Appendix}{Appendices}
\cCrefname{subappendix}{Section}{Sections}
\cCrefname{subsubappendix}{Section}{Sections}

\cCrefparenedname{equation}{(}{)}
\cCrefname{figure}{Figure}{Figures}
\cCrefname{table}{Table}{Tables}
\cCrefname{footnote}{Footnote}{Footnotes}

\cCrefname{thm}{Theorem}{Theorems}
\cCrefname{prop}{Proposition}{Propositions}
\cCrefname{cor}{Corollary}{Corollaries}
\cCrefname{lem}{Lemma}{Lemmata}
\cCrefname{defi}{Definition}{Definitions}
\cCrefname{rem}{Remark}{Remarks}
\cCrefname{ex}{Example}{Examples}
\cCrefname{q}{Question}{Questions}

\newtheorem{thm}{Theorem}[section]
\newtheorem{prop}[thm]{Proposition}
\newtheorem{cor}[thm]{Corollary}
\newtheorem{lem}[thm]{Lemma}
\newtheorem{q}[thm]{Question}

\theoremstyle{definition}
\newtheorem{defi}[thm]{Definition}
\newtheorem{ex}[thm]{Example}

\theoremstyle{remark}
\newtheorem{rem}[thm]{Remark}

\makeatletter

\@addtoreset{equation}{section}
\makeatother

\newcommand{\calA}{\mathcal{A}}
\newcommand{\calC}{\mathcal{C}}
\newcommand{\calE}{\mathcal{E}}
\newcommand{\calG}{\mathcal{G}}
\newcommand{\calP}{\mathcal{P}}
\newcommand{\calM}{\mathcal{M}}
\newcommand{\calT}{\mathcal{T}}

\newcommand{\ZZ}{\mathbb{Z}}
\newcommand{\RR}{\mathbb{R}}
\newcommand{\kk}{\Bbbk}

\newcommand{\eb}{\mathbf{e}}
\newcommand{\fkG}{\mathfrak{G}}

\newcommand{\set}[1]{\left\{ #1 \right\}}
\newcommand{\setcond}[2]{\set{#1 : #2}}
\newcommand{\rbra}[1]{\left( #1 \right)}

\DeclarePairedDelimiter{\floor}{\lfloor}{\rfloor}
\DeclarePairedDelimiter{\card}{\lvert}{\rvert}

\DeclareMathOperator{\conv}{conv}
\let\int\relax
\DeclareMathOperator{\int}{int}

\begin{document}

\title{Pseudo-Gorenstein edge rings and a new family of almost Gorenstein edge rings}
\author{Yuta Hatasa, Nobukazu Kowaki, Koji Matsushita}

\address[Y. Hatasa]{
    Department of Mathematics,
    Tokyo Institute of Technology, 2--12--1 \=Ookayama, Meguro-ku,
    Tokyo 152--8551, Japan
}
\address[Y. Hatasa]{
    International Institute for Sustainability with Knotted Chiral Meta Matter (WPI-SKCM$^2$),
    Hiroshima University, 1--3--1 Kagamiyama, Higashi-Hiroshima, Hiroshima 739--8526, Japan
}
\email{hatasa.y.aa@m.titech.ac.jp}

\address[N. Kowaki]{Department of Pure and Applied Mathematics, Graduate School of Information Science and Technology, Osaka University, Suita, Osaka 565-0871, Japan}
\email{u793177f@ecs.osaka-u.ac.jp}

\address[K. Matsushita]{Department of Pure and Applied Mathematics, Graduate School of Information Science and Technology, Osaka University, Suita, Osaka 565-0871, Japan}
\email{k-matsushita@ist.osaka-u.ac.jp}


\begin{abstract}
    In this paper, we study edge rings and their $h$-polynomials.
    We investigate when edge rings are pseudo-Gorenstein, which means that the leading coefficients of the $h$-polynomials of edge rings are equal to 1.
    Moreover, we compute the $h$-polynomials of a special family of edge rings and show that some of them are almost Gorenstein.
\end{abstract}

\maketitle

\section{Introduction}

\subsection{Backgrounds}

Let $R = \bigoplus_{k\ge 0} R_k$ be a Cohen--Macaulay homogeneous domain of dimension $d$ over an algebraically closed field $R_0 = \kk$ with characteristic $0$.
Then the \textit{Hilbert series} of $R$ is defined as the formal power series $\sum_{k\ge 0}(\dim_\kk R_k)t^k$, and it is known that we can write the Hilbert series of $R$ as the following form:
\[
    \sum_{k \ge 0} (\dim_\kk R_k) t^k=\frac{h_0+h_1t+\cdots+h_st^s}{(1-t)^d},
\]
where $h_s \neq 0$.
We call the polynomial $h_0+h_1t+\cdots+h_st^s$ the \textit{$h$-polynomial} of $R$, denoted by $h(R;t)$, and call the sequence $(h_0,h_1,\ldots,h_s)$ the \textit{$h$-vector} of $R$.
In addition, we also call the index $s$ the \textit{socle degree} of $R$, and denoted by $s(R)$.

The $h$-polynomial ($h$-vector) of $R$ is one of the most important invariants in the theory of commutative algebra because it tells us what commutative-algebraic properties $R$ has.
In fact, it is well known that $R$ is Gorenstein if and only if $R$ has the symmetric $h$-vector, that is, $h_i=h_{s(R)-i}$ holds for each $i=0,\ldots,\floor{s(R)/2}$ (\cite{stanley1978hilbert}).
Moreover, recent studies have shown that the $h$-vector has connections not only to Gorensteinness, but also to its generalized notions such as levelness (\cite{stanley1976cohen}), almost Gorensteinness (\cite{goto2015almost}) and nearly Gorensteinness (\cite{herzog2019trace}).
For example, the following facts are known:
\begin{itemize}
    \item $R$ is level if and only if the Cohen--Macaulay type of $R$ is equal to $h_{s(R)}$ (see \cite[3.2~Proposition]{stanley1996comb}). This implies that if $R$ is level, then $h_{s(R)}=1$ if and only if $R$ is Gorenstein;
    \item If $s(R)\ge 2$ and $R$ is almost Gorenstein, then $h_{s(R)}=1$ (\cite[Theorem~4.7]{higashitani2016almost});
    \item If $R$ is nearly Gorenstein, then $h_{s(R)}=1$ if and only if $R$ is Gorenstein (\cite[Theorem~3.6]{miyashita2024levelness}).
\end{itemize}
These show that the leading coefficient $h_{s(R)}$ of the $h$-polynomial of $R$ has wealth information.
In particular, these motivate determining when $h_{s(R)}=1$.
A Cohen--Macaulay homogeneous ring $R$ is called \textit{pseudo-Gorenstein} if $h_{s(R)}=1$ (\cite{ene2015pseudo}), and certain classes of pseudo-Gorenstein rings are characterized (see, e.g., \cite{ene2015pseudo,rinaldo2023level2,rinaldo2023level}).

In this paper, we study the $h$-polynomials and pseudo-Gorensteinness of certain homogeneous domains, called \textit{edge rings}.

\medskip

For a finite simple graph $G$ on the vertex set $V(G)=[d] \coloneqq \{1,\ldots,d\}$ with the edge set $E(G)$, we define the \textit{edge ring} of $G$ as follows:
\[
    \kk[G] \coloneqq \kk[t_i t_j : \{i,j\}\in E(G)]\subset \kk[t_1,\ldots,t_d].
\]

Edge rings began to be studied by Ohsugi–Hibi (\cite{ohsugi1998normal}) and Simis–Vasconcelos–Villarreal (\cite{simis1998integral}).
Since then, many researchers have studied the commutative ring-theoretic properties of edge rings.
In particular, the $h$-polynomials of the edge rings have been investigated.
As far as we know, the $h$-vectors of the edge rings of the following graphs have been computed:
\begin{itemize}
    \item Complete graphs and complete bipartite graphs (see \cite[Section~10.6]{villarreal2001monomial});
    \item Bipartite graphs (via interior polynomials) (\cite{kalman2017root});
    \item A family of graphs composed of a complete bipartite graph and a cycle graph (\cite{galetto2019betti});
    \item A family of graphs consisting of odd cycles that share a single common vertex (\cite{bhaskara2023h,higashitani2023h}).
\end{itemize}
Moreover, Gorensteinness, leveleness, almost Gorensteinness, nearly Gorensteinness and pseudo-Gorensteinness of edge rings have been investigated, respectively (see \cite{ohsugi2006GorEdge,higashitani2022levelness,higashitani2023h,bhaskara2023h,hall2023nearly,guan2022coefficients}).

\medskip

There are two goals of this paper; one is to characterize when edge rings are pseudo-Gorenstein.
Another one is to find a new family of edge rings having commutative ring-theoretic properties mentioned above and to examine the behavior of their $h$-polynomials.
In this paper, we focus on almost Gorenstein edge rings and their $h$-polynomials.

\bigskip

\subsection{Results}
First, we completely characterize when the edge rings of bipartite graphs are pseudo-Gorenstein:

\begin{thm}[{\Cref{thm:char_hs=1} and \Cref{cor:char_hs=1}}]\label{main1}
    Let $G$ be a bipartite graph.
    Then the edge ring $\kk[G]$ is pseudo-Gorenstein if and only if every block of $G$ is matching-covered.
\end{thm}
\noindent We also derive some corollaries from \Cref{main1} (\Cref{cor:hamilton,cor:biparregular,cor:evenver}).

Moreover, we investigate the case of non-bipartite graphs in \Cref{subsec:non-bip}.
This case becomes much more difficult than the bipartite graph case, and the assertions of \Cref{main1} and corollaries mentioned above are no longer true.
However, the assertions can still be ensured with the addition of certain assumptions (\Cref{thm:regular_hs=1,thm:MatchingNotBipar}).

\medskip

Next, we give a new family of almost Gorenstein edge rings for which their $h$-vectors are computed.
For integers $m, n \ge 3$ and $0 \le r \le \min\{m, n\}$, let $G_{m, n, r}$ be the graph obtained from a
complete bipartite graph $K_{m, n}$ by removing a matching $M$ with $\card{M} = r$ (see \Cref{subsec:new_family} for the precise definition of $G_{m,n,r}$).

\begin{thm}[{\Cref{prop:hilbG} and \Cref{thm:almGnr}}]\label{main2}
    We have
    \[
        h_{\kk[G_{m, n, r}]}(t) = 1 +\rbra{(m - 1)(n - 1) - r} t
        + \sum_{i = 2}^{\min\{m, n\}} \binom{m - 1}{i} \binom{n - 1}{i} t^i.
    \]
    Moreover, $\kk[G_{m, n, r}]$ is almost Gorenstein if and only if $m = n$.
\end{thm}

\bigskip

\subsection{Organization}

The structure of this paper is as follows.
In \Cref{sec:pre}, we prepare the required materials for the discussions later.
We recall definitions and notations associated with commutative algebras and edge rings.
In \Cref{sec:leading}, we discuss the pseudo-Gorensteinness of edge rings.
Before that, we recall some definitions and facts on graph theory.
We then provide the characterization of pseudo-Gorenstein edge rings for bipartite graphs and investigate the case of non-bipartite graphs.
In \Cref{sec:almost}, we introduce a new family of almost Gorenstein edge rings and compute their $h$-polynomials.
Moreover, we give an observation on the behavior of the $h$-polynomials of almost Gorenstein edge rings.

\bigskip

\subsection*{Acknowledgement}
The author would like to thank Akihiro Higashitani, Tam\'as K\'alm\'an, and Sora Miyashita for their helpful comments and advice on improving this paper.
The third author is partially supported by Grant-in-Aid for JSPS Fellows Grant JP22J20033.

\section{Preliminaries}\label{sec:pre}

\subsection{Cohen--Macaulay homogeneous domains and their $h$-polynomials}
Throughout this subsection, let $R = \bigoplus_{k\ge 0} R_k$ be a Cohen--Macaulay homogeneous domain of dimension $d$ over an algebraically closed field $R_0 = \kk$ with characteristic $0$ and let $(h_0,\ldots,h_s)$ be the $h$-vector.
First, we introduce some fundamental objects associated with homogeneous rings (consult, e.g., \cite{bruns1998cohen} for the introduction).
\begin{itemize}
    \item Let $\omega_R$ denote a canonical module of $R$ and let $a(R)$ denote the $a$-invariant of $R$, i.e., $a(R)=-\min\{j:(\omega_R)_j \neq 0\}$.
          When it is clear from context, we simply write $a$ instead of $a(R)$.
    \item For a graded $R$-module $M$, we use the following notation:
          \begin{itemize}
              \item Let $\mu(M)$ denote the number of minimal generators of $M$.
              \item Let $e(M)$ denote the multiplicity of $M$. Then the inequality $\mu(M) \leq e(M)$ always holds.
              \item Let $M(-\ell)$ denote the $R$-module whose grading is given by $M(-\ell)_k=M_{k-\ell}$ for any $k \in \ZZ$.
          \end{itemize}
    \item As mentioned in the introduction, we denote the $h$-polynomial (resp. the socle degree) of $R$ by $h(R;t)$ (resp. $s(R)$).
          We simply write $s$ instead of $s(R)$ as in the $a$-invariant.
          Note that $h_s=\dim_\kk (\omega_R)_{-a}$ and $d+a=s$ (see \cite[Section 4.4]{bruns1998cohen}).
\end{itemize}

Next, we recall the definitions and some properties of pseudo-Gorenstein rings and almost Gorenstein rings, respectively.

\begin{defi}[{\cite[Section~1]{ene2015pseudo}}]
    We call $R$ \textit{pseudo-Gorenstein} if $h_{s}=\dim_\kk (\omega_R)_{-a}=1$.
\end{defi}

\begin{defi}[{\cite[Definition 1.5]{goto2015almost}}]
    We call $R$ \textit{almost Gorenstein}
    if there exists an exact sequence of graded $R$-modules
    \begin{equation}\label{eq:exact}
        0 \to R \to \omega_R(-a) \to C \to 0
    \end{equation}
    with $\mu(C)=e(C)$.
\end{defi}
Under our assumptions on $R$, there always exists a degree-preserving injection from $R$ to $\omega_R(-a)$ (\cite[Proposition 2.2]{higashitani2016almost}).
Moreover, we see that $\mu(C)=\mu(\omega_R)-1$ and $e(C)=\sum_{j=0}^{s-1}\rbra{\rbra{h_s+\cdots+h_{s-j}}-\rbra{h_0+\cdots+h_j}}$ (\cite[Propositions~2.3 and 2.4]{higashitani2016almost}), especially, $\mu(C)$ and $e(C)$ do not depend on $C$.

Let $\widetilde{e}(R) \coloneqq \sum_{j=0}^{s-1}\rbra{\rbra{h_s+\cdots+h_{s-j}}-\rbra{h_0+\cdots+h_j}}$. Note that $\widetilde{e}(R)\ge \mu(\omega_R)-1$.

\begin{thm}[{\cite[Corollary~2.7]{higashitani2016almost}}]\label{thm:alm}
    Work with the same notation as above.
    Then $R$ is almost Gorenstein if and only if one has $\widetilde{e}(R)=\mu(\omega_R)-1$.
\end{thm}

\begin{thm}[{\cite[Theorem 4.7]{higashitani2016almost}}]\label{thm:almh_s=1}
    If $R$ is almost Gorenstein and $s(R) \ge 2$, then $h_{s(R)}=1$, that is, $R$ is pseudo-Gorenstein.
\end{thm}

\bigskip

\subsection{Edge rings and edge polytopes}
Throughout this paper, all graphs are finite and have no loops and no multiple edges.
We recall the definition and some properties of edge rings.
We also recall edge polytopes, which are lattice polytopes arising from graphs.
We can regard edge rings as affine semigroup rings associated with edge polytopes.
See, e.g., \cite[Section 10]{villarreal2001monomial} or \cite[Section 5]{herzog2018binomial} for the introduction to edge rings.

For a graph $G$ on the vertex set $V(G)=[d]$ with the edge set $E(G)$, we consider the morphism $\psi_G$ of $\kk$-algebras:
\[
    \psi_G : \kk[x_{\{i,j\}} : \{i,j\}\in E(G)] \to \kk[t_1,\ldots,t_d], \text{ induced by } \psi_G(x_{\{i,j\}})=t_it_j.
\]
Then we denote the image (resp. the kernel) of $\psi_G$ by $\kk[G]$ (resp. $I_G$) and call $\kk[G]$ the \textit{edge ring} of $G$.
The edge ring $\kk[G]$ is a homogeneous domain by setting $\deg(t_it_j)=1$ for each $\set{i,j}\in E(G)$.

Given an edge $e=\{i,j\} \in E(G)$, let $\rho(e) \coloneqq \eb_i+\eb_j$, where $\eb_i$ denotes the $i$-th unit vector of $\RR^d$ for $i=1,\ldots,d$.
We define the convex polytope associated to $G$ as follows:
\[
    P_G \coloneqq \conv(\{\rho(e) : e \in E(G)\}) \subset \RR^d.
\]
We call $P_G$ the \textit{edge polytope} of $G$.

We can regard $\kk[G]$ as an affine semigroup ring as follows:
Let $\calA_G=\{\rho(e) : e \in E(G)\}$ and let $S_G$ be the affine semigroup generated by $\calA_G$, that is, $S_G=\ZZ_{\ge 0}\calA_G$. Then we can easily see that $\kk[G]$ is isomorphic to the affine semigroup ring of $S_G$.

Let $b(G)$ be the number of bipartite connected components of $G$, then we have $\dim P_G=d-b(G)-1$ (see \cite[Proposition~10.4.1]{villarreal2001monomial}).
This implies that the Krull dimension of $\kk[G]$ is equal to $d-b(G)$.

\begin{thm}[cf. {\cite{ohsugi1998normal,simis1998integral}}]
    Let $G$ be a graph. Then the following are equivalent:
    \begin{enumerate}
        \item $\kk[G]$ is normal;
        \item $S_G=\RR_{\ge 0}\calA_G\cap \ZZ\calA_G$;
        \item $\dim_\kk \kk[G]_k =\card{kP_G\cap \ZZ^d}$ for any $k\in \ZZ_{\ge 0}$.
        \item each connected component of $G$ satisfies the {\upshape odd cycle condition}, i.e., for each pair of odd cycles $C$ and $C'$ with no common vertex, there is an edge $\{v,v'\}$ with $v \in V(C)$ and $v' \in V(C')$
    \end{enumerate}
\end{thm}
For a lattice polytope $P\subset \RR^d$ and $k\in \ZZ_{\ge 0}$, $L_P(k) \coloneqq \card{kP\cap \ZZ^d}$ is called the \textit{Ehrhart polynomial} of $P$.
If $\kk[G]$ is normal, we can compute the Ehrhart polynomial of $P_G$ using the $h$-vector of $\kk[G]$ as follows:
\begin{prop}[{see \cite[Lemma~3.14]{beck2007computing}}]
    \label{prop:326644C5}
    Let $G$ be a graph with $\dim P_G=\delta$ and let $(h_0,\ldots,h_s)$ be the $h$-vector of $\kk[G]$.
    If $\kk[G]$ is normal, then the Ehrhart polynomial $L_{P_G}(k)$ of $P_G$ can be described as follows:
    \[
        L_{P_G}(k)=h_0\binom{k+\delta}{\delta}+h_1\binom{k+\delta-1}{\delta}+\cdots +h_s\binom{k+\delta-s}{\delta}.
    \]
\end{prop}

Moreover, if $\kk[G]$ is normal, then $\kk[G]$ is Cohen--Macaulay and the ideal generated by monomials corresponding to the elements in $\int(S_G)$, where $\int(S_G)$ denotes the set of elements of $S_G$ in the relative interior of the cone generated by $S_G$, is the canonical module of $\kk[G]$ (see \cite[Theorems~6.10 and 6.31]{bruns2009polytopes}).
In particular, let $\ell_G \coloneqq \min\set{n : \int(nP_G)\cap \ZZ^d\neq \emptyset}$, where $\int(P)$ denotes the relative interior of a polytope $P$, then we have $\ell_G=-a(\kk[G])$.

To observe $\int(S_G)$, we give a description of the cone $\RR_{\ge 0}\calA_G$.
To this end, we recall some terminologies and notations.
We say that a subset $T$ of $V(G)$ is an \textit{independent set} of $G$ if $\{v,w\}\notin E(G)$ for any $v,w\in T$.
Note that each singleton is regarded as an independent set.
We denote the set of non-empty independent sets of $G$ by $\calT_G$.

For $i \in [d]$ and $T \in \calT_G$, let
\begin{align*}
     & H_i \coloneqq \{(x_1,\ldots,x_d) \in \RR^d : x_i= 0\},
     & H_T \coloneqq \set{(x_1,\ldots,x_d) \in \RR^d : \sum_{j \in N_G(T)}x_j - \sum_{i \in T}x_i = 0},     \\
     & H_i^+ \coloneqq \{(x_1,\ldots,x_d) \in \RR^d : x_i\ge 0\},
     & H_T^+ \coloneqq \set{(x_1,\ldots,x_d) \in \RR^d : \sum_{j \in N_G(T)}x_j - \sum_{i \in T}x_i \ge 0}, \\
     & H_i^{>} \coloneqq \{(x_1,\ldots,x_d) \in \RR^d : x_i > 0\},
     & H_T^{>} \coloneqq \set{(x_1,\ldots,x_d) \in \RR^d : \sum_{j \in N_G(T)}x_j - \sum_{i \in T}x_i > 0},
\end{align*}
where $N_G(T) \coloneqq \set{v\in [d] : \set{v,w}\in E(G) \text{ for some }w\in T}$.

The following terminologies are used in \cite{ohsugi1998normal}: Suppose that $G$ is connected.
\begin{itemize}
    \item For a subset $W \subset V(G)$, let $G\setminus W$ denote the induced subgraph with respect to $V(G)\setminus W$ (for a vertex $v$, we denote by $G \setminus v$ instead of $G \setminus \{v\}$).
    \item We call a vertex $v$ of $G$ \textit{ordinary} if $G \setminus v$ is connected.
    \item Given an independent set $T \subset V(G)$, let $B(T)$ denote the bipartite graph on $T \cup N_G(T)$ with the edge set $\{\{v,w\}\in E(G) : v \in T, w \in N_G(T)\}$.
    \item When $G$ is a bipartite graph with the partition $V(G)=V_1\sqcup V_2$, a non-empty set $T \subset V_1$ is said to be an \textit{acceptable set} if the following conditions are satisfied:
          \begin{itemize}
              \item $B(T)$ is connected;
              \item $G \setminus V(B(T))$ is a connected graph with at least one edge.
          \end{itemize}
\end{itemize}

\begin{thm}[{\cite[Theorem~1.7]{ohsugi1998normal}, see also \cite[Section~10.7]{villarreal2001monomial}}]\label{facet}
    Let $G$ be a connected graph.
    Then the cone $\RR_{\ge 0}\calA_G$ has the following representation:
    \[
        \RR_{\ge 0}\calA_G=\rbra{\bigcap_{i\in [d]}H^+_i}\cap \rbra{\bigcap_{T\in \calT_G}H^+_T}.
    \]
    Moreover, if $G$ is bipartite, each facet of $P_G$ is defined by a supporting hyperplane $H_i$ for some ordinary vertex $i$ or $H_T$ for some acceptable set $T$.
\end{thm}

For a subset $F\subset E(G)$, we set $v_F \coloneqq \sum_{e\in F}\rho(e)$.
The following lemma is obvious but important to our results:
\begin{lem}\label{lem:indepe}
    Let $G$ be a graph.
    For $F\subset E(G)$ and $T\in \calT_G$, we have $v_F \in H_T^>$ if there exists $f\in F$ with $f\in E(G)\setminus E(B(T))$ and $f\cap N_G(T)\neq \emptyset$.
\end{lem}

A connected graph $G$ is said to be \textit{$2$-connected} if all vertices of $G$ are ordinary.
A \textit{block} is a maximal $2$-connected component of $G$.

\begin{prop}[{\cite[Proposition 10.1.48]{villarreal2001monomial}}]\label{product}
    Let $G$ be a bipartite graph and let $B_1,\ldots, B_n$ be their blocks.
    Then we have $\kk[G]\cong \kk[B_1]\otimes \cdots \otimes \kk[B_n]$.
\end{prop}

\bigskip

\section{Pseudo-Gorenstein edge rings}\label{sec:leading}
In this section, we study the pseudo-Gorensteinness of edge rings.

\subsection{Preliminary on graph theory}
Before that, we need some notions and notations on (directed) graphs.
Let $G$ be a $2$-connected graph.
Then $G$ has an \textit{(open) ear decomposition} (cf. \cite[Proposition~3.1.1]{diestel2017graph}), i.e., $G$ can be decomposed as $C\cup P_1\cup\cdots\cup P_r$ where $C$ is a cycle, $P_i$ is a path and $(C\cup P_1 \cup \cdots \cup P_{i-1})\cap P_i$ consists of end vertices of $P_i$ for each $i$.
Let $\phi(G)$ denote the minimum number of even paths (containing the first cycle) in an ear decomposition of $G$.
Note that $\phi(G)\ge 1$ if $G$ is bipartite since the first cycle must be an even
cycle.
We call an ear decomposition of $G$ \textit{optimal} if the number of even paths is just $\phi(G)$.
Let $G$ be a $2$-connected bipartite graph and fix an optimal ear decomposition $G=C\cup P_1\cup\cdots\cup P_r$ with
\begin{equation}\label{oed}
    \begin{split}
         & V(P_i)=\set{p_{i,0},p_{i,1},\ldots,p_{i,m_i}}, \quad V(C)=\set{c_1,\ldots,c_{2n}},                      \\
         & E(P_i)=\set{\set{p_{i,0},p_{i,1}},\set{p_{i,1},p_{i,2}},\ldots,\set{p_{i,m_i-1},p_{i,m_i}}} \text{ and} \\
         & E(C)=\set{\set{c_1,c_2},\set{c_2,c_3},\ldots,\set{c_{2n-1},c_{2n}},\set{c_{2n},c_1}}.
    \end{split}
\end{equation}
We set
\begin{equation}\label{edgeset}
    \begin{split}
         & E_i \coloneqq
        \begin{cases*}
            \{\{p_{i,1},p_{i,2}\},\{p_{i,3},p_{i,4}\},\cdots,\{p_{i,m_i -2},p_{i,m_i-1}\}\} & if $m_i$ is odd,  \\
            \{\{p_{i,1},p_{i,2}\},\{p_{i,3},p_{i,4}\},\cdots,\{p_{i,m_i-1},p_{i,m_i}\}\}    & if $m_i$ is even,
        \end{cases*} \\
         & E_C \coloneqq \{\{c_1,c_2\},\{c_3,c_4\},\cdots,\{c_{2n-1},c_{2n}\}\} \text{ and}                 \\
         & \calE \coloneqq E_C\cup\rbra{\bigcup_{i=1}^rE_i}.
    \end{split}
\end{equation}
Note that $E_i=\emptyset$ if $m_i=1$, and that $\card{\calE}=(\phi(G)+\card{V(G)}-1)/2$.

\medskip

The following terminologies and facts are mentioned in \cite[Section~5]{frank1993conservative}:
For a directed graph $D$,
\begin{itemize}
    \item the set of edges entering a subset $W\subset V(D)$ is called a \textit{directed cut} if no edges leave $W$.
    \item A directed graph $D$ is \textit{strongly connected} if there is a directed path from every vertex to any other.
    \item The following are equivalent:
          \begin{enumerate}
              \item $D$ is strongly connected;
              \item $D$ does not contain directed cuts;
              \item $D$ has a so-called \textit{directed ear decomposition}, which is an ear decomposition of $D$, denoted as \Cref{oed}, such that each edge $\set{c_i,c_{i+1}}$ (resp. $\set{p_{i,j},p_{i,j+1}}$) is directed from $c_i$ to $c_{i+1}$ (resp. $p_{i,j}$ to $p_{i,j+1}$).
          \end{enumerate}
\end{itemize}

For a $2$-connected bipartite graph $G$ with the partition $V(G)=V_1\sqcup V_2$, let $G'$ be the directed graph obtained from $G$ by orienting each edge from $V_1$ to $V_2$
and let $\widetilde{G}'$ be the directed graph obtained by reversing the orientation of edges $\calE$ of $G'$.
We can see that $\widetilde{G}'$ has a directed ear decomposition.
Let $\sigma(G)$ denote the minimum cardinality of a subset of $E(G)$ that contains at least one edge of each directed cut of $G'$.

According to \cite{frank1993conservative} and \cite{valencia2003canonical}, we have the following:

\begin{thm}[{\cite[Theorems~4.5 and 5.3]{frank1993conservative}, \cite[Proposition~4.2]{valencia2003canonical}}]\label{thm:relation}
    Let $G$ be a $2$-connected bipartite graph.
    Then we have
    \[
        \ell_G=\sigma(G)=\card{\calE}=\frac{\phi(G)+\card{V(G)}-1}{2}.
    \]
\end{thm}

A graph $G$ is called \textit{matching-covered} if it is connected and each edge is contained in some perfect matching.
There are many characterizations of matching-covered bipartite graphs (see, e.g., \cite[Section~4.1]{lovasz2009matching} or \cite[Section~6.1]{asratian1998bipartite}), of which the following is the most important one for this paper:

\begin{thm}[cf. {\cite[Theorems~6.1.5 and 6.1.6]{asratian1998bipartite}}]\label{matching-covered}
    Let $G$ be a connected bipartite graph with the partition $V(G)=V_1\sqcup V_2$.
    Then the following are equivalent:
    \begin{enumerate}
        \item $G$ is matching-covered;
        \item $\card{V_1}=\card{V_2}$ and $\card{N_G(T)}>\card{T}$ for every non-empty subset $T\subsetneq V_1$;
        \item $G$ is $2$-connected and $\phi(G)=1$.
    \end{enumerate}
\end{thm}

\begin{rem}
    If $G$ is not bipartite, then (i) and (iii) in \Cref{matching-covered} are not equivalent, but the implication ``(i) $\Longrightarrow$ (iii)'' is true in general (see \cite[Section~5.4]{lovasz2009matching}).
\end{rem}

For a graph $G$, we say that $G$ is $k$-\textit{regular} if all vertices of $G$ have the same degree $k$, and $G$ is \textit{regular} if $G$ is $k$-regular for some $k$.

\begin{lem}
    \label{lem:B(T)}
    Let $G$ be a connected $k$-regular graph.
    Then for an independent set $T\subset V(G)$,
    one has $\card{T} = \card{N_G(T)}$ if and only if $G = B(T)$.
\end{lem}

\begin{proof}
    Since $G$ is $k$-regular, we have
    \begin{equation}\label{ineq_reg}
        k \card{T} = \sum_{v \in N_G(T)} \deg_{B(T)}(v) \le \sum_{v \in N_G(T)} k = k \card{N_G(T)},
    \end{equation}
    where $\deg_{B(T)}(v)$ is the degree of $v$ in the subgraph $B(T)$.
    If $\card{T} = \card{N_G(T)}$, then the equality of \Cref{ineq_reg} holds, which implies that $G=B(T)$ since $G$ is connected.

    Conversely, assume that $G = B(T)$.
    In this case, $G$ is bipartite with the partition $V(G) = T \sqcup N_G(T)$.
    Moreover, the equality of \Cref{ineq_reg} holds, so we have
    $\card{T} = \card{N_G(T)}$.
\end{proof}

\bigskip

\subsection{The case of bipartite graphs}\label{subsec:bipar}
In this subsection, we investigate when the edge rings of bipartite graphs are pseudo-Gorenstein.
Let $G$ be a bipartite graph with the partition $V_1\sqcup V_2$ and let $B_1,\ldots,B_m$ be the blocks of $G$.
Then we have $h_{s(\kk[G])}=\prod_{i=1}^m h_{s(\kk[B_i])}$ from \Cref{product}, so it is enough to study the case when $G$ is $2$-connected.
In what follows, we assume $G$ is $2$-connected, fix an optimal ear decomposition described as \Cref{oed} and let $\calE$ be the subset of $E(G)$ defined in \Cref{edgeset}.

We first give the following lemma:

\begin{lem}\label{lem:interior}
    We have $v_\calE \in \int(S_G)$.
\end{lem}
\begin{proof}
    From \Cref{facet}, it is enough to show that $v_\calE \in H_i^>$ for each $i\in [d]$ and $v_\calE \in H_T^>$ for any acceptable set $T\subsetneq V_1$.
    Since $\bigcup_{e\in \calE}e=V(G)$,
    we have $v_\calE \in H_i^>$ for any $i\in V(G)$.
    For an acceptable set $T\subsetneq V_1$, let $W \coloneqq T\cup N_G(T)$ and $\calC \coloneqq \set{\set{v,w} \in E(G): v\in V_1\setminus T \text{ and } w\in N_G(T)}$.
    Then $\calC$ is a directed cut of $G'$, but is not a directed cut of $\widetilde{G}'$ since $\widetilde{G}'$ has a directed ear decomposition, and hence $\widetilde{G}'$ does not contain directed cuts.
    This implies that $\calC\cap \calE\neq \emptyset$, and hence $v_\calE \in H_T^>$ from \Cref{lem:indepe}.
\end{proof}

We now present the following main theorem in this subsection:

\begin{thm}\label{thm:char_hs=1}
    Let $G$ be a $2$-connected bipartite graph.
    Then the following are equivalent:
    \begin{enumerate}
        \item $\kk[G]$ is pseudo-Gorenstein;
        \item $\phi(G)=1$;
        \item $G$ is matching-covered.
    \end{enumerate}
\end{thm}

\begin{proof}
    It follows from \Cref{matching-covered} that (ii) and (iii) are equivalent, so it is enough to show that (i) and (ii) are equivalent.

    First, we show that (i) implies (ii).
    Suppose that $\phi(G)\ge 2$.
    Then there exists $i$ such that $P_i$ has the even length.
    Let $E'_i \coloneqq \set{\set{p_{i,0},p_{i,1}},\set{p_{i,2},p_{i,3}},\ldots,\set{p_{i,m_i-2},p_{i,m_i-1}}}$
    and let $\calE' \coloneqq (\calE\setminus E_i)\cup E'_i$.
    We can see that $v_\calE \neq v_{\calE'}$ and $v_\calE,v_{\calE'}\in \int(S_G)$ from \Cref{lem:interior}.
    In particular, we have $v_\calE,v_{\calE'}\in \int(\ell_G P_G)\cap \ZZ^d$
    since $\card{\calE}=\card{\calE'}=\ell_G$ from \Cref{thm:relation}.
    Therefore, we get $h_s=\dim (\omega_{\kk[G]})_{\ell_G}=\card{\int(\ell_G P_G)\cap \ZZ^d}\ge 2$.

    Next, we prove that (ii) implies (i).
    Actually, this has been already shown in \cite[Theorem~32]{guan2022coefficients} via interior polynomials, but we will provide a short self-contained proof of it.
    In this case, we can easily see that $v_\calE=(1,\ldots,1)$ and this is the unique lattice point in $\int(\ell_G P_G)$.
    Thus, we have $h_s=1$.
\end{proof}

\Cref{thm:char_hs=1} gives some corollaries.

\begin{cor}\label{cor:char_hs=1}
    Let $G$ be a bipartite graph and let $B_1,\ldots, B_m$ be the blocks of $G$.
    Then $\kk[G]$ is pseudo-Gorenstein if and only if $B_i$ is matching-covered for each $i=1,\ldots,m$.
\end{cor}

\begin{proof}
    It follows immediately from $h_{s(\kk[G])}=\prod_{i=1}^m h_{s(\kk[B_i])}$.
\end{proof}

Let $G$ be a graph.
A cycle containing all the vertices of $G$ is said to be a \textit{Hamilton cycle}, and a graph containing a Hamilton cycle is said to be \textit{Hamiltonian}.

\begin{cor}\label{cor:hamilton}
    Let $G$ be a Hamiltonian bipartite graph.
    Then $\kk[G]$ is pseudo-Gorenstein.
\end{cor}

\begin{proof}
    In this situation, we have $\phi(G)=1$.
    Indeed, $G$ has an optimal ear decomposition $C\cup P_1\cup \cdots P_r$ where $C$ is a Hamilton cycle of $G$ and $P_i$ is a path of length $1$.
\end{proof}

\begin{cor}\label{cor:biparregular}
    Let $G$ be a connected regular bipartite graph with the partition $V_1\sqcup V_2$.
    Then $\kk[G]$ is pseudo-Gorenstein.
\end{cor}

\begin{proof}
    From \Cref{matching-covered}, it is enough to show that $\card{V_1}=\card{V_2}$ and $\card{N_G(T)}>\card{T}$ for every non-empty subset $T\subsetneq V_1$.
    The fact that $G$ is regular implies that $G$ has a perfect matching (and hence $\card{V_1}=\card{V_2}$) and that $\card{N_G(T)}\ge \card{T}$ for every subset $T\subset V_1$ (\cite[see, e.g., Corollaries~6.1.3 and 6.1.4]{asratian1998bipartite}).
    If $\card{N_G(T)}=\card{T}$ holds, then we have $T=V_1$ from \Cref{lem:B(T)}, thus $\card{N_G(T)}>\card{T}$ for every non-empty subset $T\subsetneq V_1$.
\end{proof}

\begin{cor}\label{cor:evenver}
    Let $G$ be a $2$-connected bipartite graph with the partition $V_1\sqcup V_2$.
    If $\kk[G]$ is pseudo-Gorenstein, then we have $\card{V_1}=\card{V_2}$ (in particular, $\card{V(G)}$ is even).
\end{cor}

\begin{proof}
    It follows immediately from \Cref{matching-covered,thm:char_hs=1}.
\end{proof}

We summarize the results of \Cref{thm:char_hs=1,cor:hamilton,cor:biparregular,cor:evenver} in the following figure:

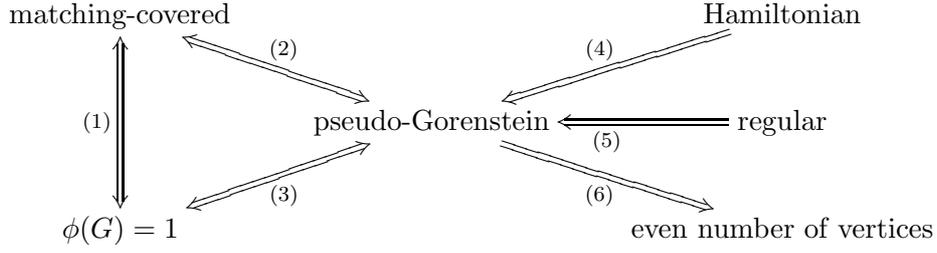
\begin{figure}[H]
    \[
        \xymatrix{
            \text{matching-covered} \ar@{<=>}[dd]_{(1)} \ar@{<=>}[dr]^{(2)} &  & \text{Hamiltonian} \ar@{=>}[dl]_{(4)} \\
            & \text{pseudo-Gorenstein}  & \text{regular} \ar@{=>}[l]^{(5)} \\
            \phi(G) = 1 \ar@{<=>}[ur]_{(3)} &  & \text{even number of vertices} \ar@{<=}[ul]^{(6)}
        }
    \]
    \caption{Implications for $2$-connected bipartite graph $G$}
    \label{Implication}
\end{figure}

\bigskip

\subsection{The case of non-bipartite graphs}\label{subsec:non-bip}
In this subsection, we investigate the pseudo-Gorensteinness of the edge rings of non-bipartite graphs.
Actually, most of the implications in \Cref{Implication} are no longer true.
We first give their counterexamples in the case of $2$-connected non-bipartite graphs satisfying the odd cycle condition.

\begin{ex}\label{counter}
    (i) Let $H_1$ be the graph on the vertex set $V(H_1)=\set{1,2,3,4}$ with the edge set $E(H_1)=\set{\set{1,2},\set{2,3},\set{3,4},\set{1,4},\set{1,3}}$ (see \Cref{counter2}).
    This graph is not matching-covered since there is no perfect matching containing $\set{1,3}$.
    On the other hand, we can compute $h(\kk[H_1];t)=1+t$, so $H_1$ is a counterexample to the implication  ``pseudo-Gorenstein $\Longrightarrow$ matching-covered''.

    \medskip

    (ii) Let $H_2$ be the graph on the vertex set $V(H_2)=\set{1,2,3,4,5,6,7}$ with the edge set $E(H_2)=\set{\set{1,2},\set{2,3},\set{3,4},\set{4,5},\set{1,5},\set{1,6},\set{3,6},\set{3,7},\set{5,7}}$ (see \Cref{counter45}).
    We can see that $\phi(H_2)=2$, $\card{V(H_2)}$ is odd and $h(\kk[H_2];t)=1+2t+t^2$.
    Thus, $H_2$ is a counterexample to the implications ``pseudo-Gorenstein $\Longrightarrow$ $\phi(G)=1$'' and ``pseudo-Gorenstein $\Longrightarrow$ even number of vertices''.
    \medskip

    (iii) Let $H_3$ be the graph on the vertex set $V(H_3)=\set{1,2,3,4,5,6}$ with the edge set $E(H_3)=\set{\set{1,2},\set{2,3},\set{3,4},\set{4,5},\set{5,6},\set{1,6},\set{1,3},\set{3,5},\set{3,6}}$ (see \Cref{counter37}).
    This graph is Hamiltonian, and
    we can see that $\phi(H_3)=1$ and $h(\kk[H_3];t)=1+3t+3t^2$.
    Therefore, $H_3$ is a counterexample to the implications ``$\phi(H_3)=1$ $\Longrightarrow$ pseudo-Gorenstein'' and ``Hamiltonian $\Longrightarrow$ pseudo-Gorenstein''.

    \medskip

    (iv) Consider the complete graph $K_5$ with five vertices.
    This graph is regular, but we can compute $h(\kk[K_5];t)=1 + 5t + 5t^2$.
    This implies that $K_5$ is a counterexample to the implication ``regular $\Longrightarrow$ pseudo-Gorenstein''.
\end{ex}

\begin{figure}[H]
    \centering
    \begin{minipage}[H]{0.32\linewidth}
        \centering
        \scalebox{0.8}{
            \begin{tikzpicture}
                \node[circle, draw] (1) at (2, 2) {1};
                \node[circle, draw] (2) at (0, 2) {2};
                \node[circle, draw] (3) at (0, 0) {3};
                \node[circle, draw] (4) at (2, 0) {4};

                \draw (1) -- (2);
                \draw (1) -- (3);
                \draw (1) -- (4);
                \draw (2) -- (3);
                \draw (3) -- (4);
            \end{tikzpicture}}
        \caption{\newline The graph $H_1$}
        \label{counter2}
    \end{minipage}
    \begin{minipage}[H]{0.32\linewidth}
        \centering
        \scalebox{0.8}{
            \begin{tikzpicture}
                \node[circle, draw] (1) at (0, 1.5) {1};
                \node[circle, draw] (2) at (1.5, 3) {2};
                \node[circle, draw] (3) at (3, 1.5) {3};
                \node[circle, draw] (4) at (4.5, 3) {4};
                \node[circle, draw] (5) at (6, 1.5) {5};
                \node[circle, draw] (6) at (1.5, 0) {6};
                \node[circle, draw] (7) at (4.5, 0) {7};

                \draw (1) -- (2) -- (3) -- (6) -- (1) -- cycle;
                \draw (3) -- (4) -- (5) -- (7) -- (3) -- cycle;
                \draw (1) .. controls (0, 3.5) .. (1.5, 3.5) -- (4.5, 3.5) .. controls (6, 3.5) .. (5);
            \end{tikzpicture}}
        \caption{\newline The graph $H_2$}
        \label{counter45}
    \end{minipage}
    \begin{minipage}[H]{0.32\linewidth}
        \centering
        \scalebox{0.8}{
            \begin{tikzpicture}
                \foreach \i in {1, 2, 3, 4, 5, 6} {
                        \node[circle, draw] (\i) at (\i * 60 : 1.5) {\i};
                    }
                \draw (1) -- (2) -- (3) -- (4) -- (5) -- (6) -- (1) -- cycle;
                \draw (1) -- (3);
                \draw (3) -- (5);
                \draw (3) -- (6);
            \end{tikzpicture}}
        \caption{\newline The graph $H_3$}
        \label{counter37}
    \end{minipage}
\end{figure}

In particular, each direction of the implication (3) in \Cref{Implication} is incorrect.
Moreover, the following example shows that the value of $\phi(G)$ does not depend on the pseudo-Gorensteinness of edge rings, that is, for any $k\in \ZZ_{>0}$, there exists a $2$-connected graph $G$ satisfying the odd cycle condition with $h_{s(\kk[G])}=1$ and $\phi(G)=k$:
\begin{ex}
    For $k\in \ZZ_{>0}$, let $\fkG_k$ be the graph on the vertex set $V(\fkG_k) \coloneqq \set{v_1,\ldots,v_{k+1}}\cup \set{u_1,\ldots,u_k} \cup \set{w_1,\ldots,w_k}$ with the edge set
    \[
        E(\fkG_k) \coloneqq \set{\set{v_1,v_{k+1}}} \cup \rbra{\bigcup_{i=1}^k \set{\set{v_i,u_i},\set{v_i,w_i},\set{u_i,v_{i+1}},\set{w_i,v_{i+1}}}}
    \]
    (see \Cref{anyphi}).
    \begin{figure}[H]
        \scalebox{0.8}{
            \begin{tikzpicture}
                \node[circle, draw] (1) at (0, 1.5) {\makebox[0.35cm]{$v_1$}};
                \node[circle, draw] (2) at (1.5, 3) {\makebox[0.35cm]{$u_1$}};
                \node[circle, draw] (3) at (1.5, 0) {\makebox[0.35cm]{$w_1$}};
                \node[circle, draw] (4) at (3, 1.5) {\makebox[0.35cm]{$v_2$}};

                \node[circle, draw] (5) at (4.5, 3) {\makebox[0.35cm]{$u_2$}};
                \node[circle, draw] (6) at (4.5, 0) {\makebox[0.35cm]{$w_2$}};
                \node[circle, draw] (7) at (6, 1.5) {\makebox[0.35cm]{$v_3$}};

                \draw (1) -- (2) -- (4) -- (3) -- (1) -- cycle;
                \draw (4) -- (5) -- (7) -- (6) -- (4) -- cycle;

                \node[circle, draw] (8) at (12.5, 3) {\makebox[0.35cm]{$u_k$}};
                \node[circle, draw] (9) at (11, 1.5) {\makebox[0.35cm]{$v_k$}};
                \node[circle, draw] (10) at (12.5, 0) {\makebox[0.35cm]{$w_k$}};
                \node[circle, draw] (11) at (14, 1.5) {\makebox[0.35cm]{$v_{k+1}$}};
                \draw (8) -- (9) -- (10) -- (11)-- (8) -- cycle;
                \draw[dashed] (7) -- ++ (0 : 2);
                \draw[dashed] (9) -- ++ (180 : 2);
                \draw (1) .. controls (0, 3.5) .. (1.5, 3.5) -- (12.5, 3.5) .. controls (14, 3.5) .. (11);
            \end{tikzpicture}
        }
        \caption{The graph $\fkG_k$}\label{anyphi}
    \end{figure}
    Then $\fkG_k$ is a $2$-connected graph satisfying the odd cycle condition with $\phi(\fkG_k)=k$.
    Indeed, any odd cycle of $\fkG_k$ contains the edge $\set{v_1,v_{k+1}}$, so $\fkG_k$ satisfies the odd cycle condition.
    Moreover, we can easily see that for any ear decomposition of $\fkG_k$, the number of even paths appearing in it is just $k$, which implies that $\phi(\fkG_k)=k$.

    We show that $\kk[\fkG_k]$ is pseudo-Gorenstein.
    From \cite[Lemmas~5.10 and 5.11]{herzog2018binomial}, we can see that
    \[
        I_{\fkG_k}=\rbra{b_i \coloneqq x_{\set{v_i,u_i}}x_{\set{w_i,v_{i+1}}}-x_{\set{u_i,v_{i+1}}}x_{\set{v_i,w_i}} : i \in [k]}.
    \]
    Since these binomials $b_i$ have no common variables, we obtain
    \[
        \kk[\fkG_k] \cong \rbra{\bigotimes_{i=1}^k \kk[x_{\set{v_i,u_i}},x_{\set{w_i,v_{i+1}}},x_{\set{u_i,v_{i+1}}},x_{\set{v_i,w_i}}]/(b_i)}[x_{\set{v_1,v_{k+1}}}]
    \]
    and $h(\kk[\fkG_k];t)=(1+t)^k$, which implies that $h_{s(\kk[\fkG_k])}=1$.
\end{ex}

The pseudo-Gorensteinness of the edge rings of non-bipartite graphs is quite different from that of bipartite graphs, and it seems difficult to completely characterize it.

We end this section by giving the following two sufficient conditions for edge rings to be pseudo-Gorenstein:
\begin{thm}\label{thm:MatchingNotBipar}
    Let $G$ be a matching-covered connected graph satisfying the odd cycle condition.
    Then $\kk[G]$ is pseudo-Gorenstein.
\end{thm}

\begin{proof}
    Our assertion holds from \Cref{thm:char_hs=1} if $G$ is bipartite, thus we may assume that $G$ is not bipartite.
    Let $v=(1,\ldots,1)$, then $v$ belongs to $S_G$ since $G$ has a perfect matching.
    We show that $v$ is in the unique lattice point in $\int(\ell_G P_G)$.
    From \Cref{facet}, it suffices to prove that $v\in H_T^>$ for any independent set $T$ of $G$.

    Since $G$ is a non-bipartite connected graph, there exists $f\in E(G)\setminus E(B(T))$ with $f\cap N_G(T)\neq \emptyset$.
    We can take a perfect matching $\calM$ including $f$ since $G$ is matching-covered, and hence $v_\calM=v\in H_T^>$ from \Cref{lem:indepe}.
\end{proof}

\begin{thm}\label{thm:regular_hs=1}
    Let $G$ be a connected regular graph with an even number of vertices
    satisfying the odd cycle condition.
    Then $\kk[G]$ is pseudo-Gorenstein.
\end{thm}

\begin{proof}
    Our assertion holds from \Cref{cor:biparregular} if $G$ is bipartite, thus we may assume that $G$ is not bipartite.
    As in the proof of \Cref{thm:MatchingNotBipar}, it is enough to show that $(1,\ldots,1) \in H_T^>$ for any independent set $T$ of $G$.
    We obtain $\card{T} \le \card{N_G(T)}$ by \Cref{ineq_reg}.
    If $\card{T} = \card{N_G(T)}$ holds, then we have $G = B(T)$ from \Cref{lem:B(T)},
    so $G$ is bipartite, which contradicts with the assumption.
    Therefore, for any independent set $T$, we have $\card{T} < \card{N_G(T)}$, and the point $(1, \ldots, 1)$ belongs to $H_T^>$.
\end{proof}

\bigskip

\section{Almost Gorenstein edge rings}\label{sec:almost}

\subsection{A new family of almost Gorenstein edge rings}\label{subsec:new_family}
In this subsection, we consider the following graph:
For integers $m, n \ge 3$ and $0 \le r \le \min\{m, n\}$, let $G_{m, n, r}$ be the graph on the vertex set $V(G_{m, n, r})=[m + n]$ with the edge set
$E(G_{m, n, r}) = \set{ \{i, j + m\} : i \in [m], j \in [n]} \setminus \set{\set{1, 1 + m}, \set{2, 2 + m}, \ldots, \set{r , r + m}}$.
We can easily see that $V(G_{m,n,r})$ has the partition $V(G_{m,n,r})=V_1\sqcup V_2$, where $V_1 \coloneqq \set{1,\ldots,m}$ and $V_2 \coloneqq \set{m+1,\ldots,m+n}$, and $G_{m,n,r}$ is bipartite.
Moreover, $G_{m,n,r}$ is $2$-connected and the subset $T\subset V_1$ is an acceptable set if and only if $T=\set{i}$ for some $i\in [r]$.

To compute the $h$-polynomial of $\kk[G_{m,n,r}]$, we recall the $h$-polynomial of the edge ring of a complete bipartite graph.
\begin{prop}[{\cite[Proposition~10.6.3]{villarreal2001monomial}}]\label{prop:00B7EDA0}
    Let $K_{m,n}$ denote the complete bipartite graph with $m+n$ vertices. Then
    \begin{equation*}
        h(\kk[K_{m,n}];t)=\sum_{i=0}^{\min\{m,n\}}\binom{m-1}{i}\binom{n-1}{i}t^i.
    \end{equation*}
\end{prop}

\begin{prop}\label{prop:hilbG}
    We have
    \[
        h(\kk[G_{m, n, r}]; t) = 1 + \rbra{(m - 1)(n - 1) - r} t
        + \sum_{i = 2}^{\min\{m, n\}} \binom{m - 1}{i} \binom{n - 1}{i} t^i.
    \]
\end{prop}

\begin{proof}
    By \Cref{prop:00B7EDA0,prop:326644C5}, it suffices to show that
    \[
        L_{P_{K_{m, n}}}(k) - L_{P_{G_{m, n, r}}}(k)
        = \card{k P_{K_{m, n}} \cap \ZZ^{m + n}} - \card{k P_{G_{m, n, r}} \cap \ZZ^{m + n}}
        = r \binom{k + m + n - 3}{m + n - 2}.
    \]
    For each $i = 1, \ldots, r$, let
    \[
        A_i \coloneqq \setcond{(x_1, \ldots, x_{m + n}) \in \ZZ^{m + n}_{\ge 0}}{x_i + x_{m + i} \ge k + 1, \sum_{j = 1}^m x_j = \sum_{j = 1}^n x_{m + j} = k}
    \]
    and $A \coloneqq A_1\cup \cdots \cup A_r$.
    Note that $A_i \cap A_j = \emptyset$ if $i \ne j$.
    We show that $A \subset k P_{K_{m, n}} \cap \ZZ^{m + n}$ and
    \[
        \card{A_i}=\binom{k + m + n -3}{m + n - 2} \text{ for each $i$, and hence }
        \card{A} = r \binom{k + m + n -3}{m + n - 2}.
    \]
    Let $S_i \coloneqq [m+n] \setminus \set{i}$ for each $i \in [r]$ and let $\sigma_i \colon S_i \to \RR^{m+n}$ be the map defined as follows:
    \[
        \sigma_i(l) \coloneqq
        \begin{cases*}
            \eb_l+\eb_{m+i} & if $1 \le l \le m$,        \\
            \eb_i+\eb_l     & if $m + 1 \le l \le m + n$
        \end{cases*}
        \text{ for $l\in S_i$}.
    \]
    Moreover, we denote the set of all $(k-1)$-element multisubsets on $S_i$ by $U_i$ and define the map $f_i \colon U_i\to \RR^{m + n}$ as
    \[
        f_i(u) \coloneqq \eb_i + \eb_{m+i} + \sum_{l\in u} \sigma(l)  \text{ for $u \in U_i$}.
    \]
    Then we can see that $f_i(U_i) = A_i$ and $f_i$ is a bijection on $A_i$.
    Therefore, we have $A_i \subset k P_{K_{m, n}} \cap \ZZ^{m + n}$ and
    $\card{A_i} = \card{U_i} = \binom{k+m+n-3}{m+n-2}$.

    It remains to show that
    $(k P_{K_{m, n}} \cap \ZZ^{m+n}) \setminus A = k P_{G_{m, n, r}} \cap \ZZ^{m+n}$.
    Any $a = (a_1, \ldots, a_{m + n}) \in A_i$ does not satisfy the inequality $\sum_{j \ne i} a_{m + j} - a_i \ge 0$ for any $i\in [r]$.
    Hence $a \notin H_{\set{i}}^{+}$, and consequently $a \notin kP_{G_{m, n, r}} \cap \ZZ^{m+n}$.
    Conversely, for any $a = (a_1, \ldots, a_{m+n}) \in (kP_{K_{m, n}} \cap \ZZ^{m+n}) \setminus A$,
    the element $a$ can be expressed as
    \[
        a = \sum_{i \in [m], j \in [n]} b_{i, j} (\eb_i + \eb_{m + j}) \text{ for some $b_{i, j} \in \ZZ_{\ge 0}$}.
    \]
    Since $(\eb_i + \eb_{m + i}) + (\eb_j + \eb_{m + j}) = (\eb_i + \eb_{m + j}) + (\eb_j + \eb_{m + i})$,
    we may assume $b_{i, i} \ge 0$ and $b_{j, j} = 0$ for $j \in [m] \setminus \{i\}$.
    Suppose $b_{i, i} > 0$. Then there exist $j, l \in [m] \setminus \{i\}$ such that $b_{j, l} > 0$ since $a_i + a_{m + i} \le k$,
    and we have $(\eb_i + \eb_{m + i}) + (\eb_j + \eb_{m + l}) = (\eb_i + \eb_{m + l}) + (\eb_j + \eb_{m + i})$.
    Hence we may assume $b_{i, m+i} = 0$, and therefore we have $a \in kP_{G_{m, n, r}}\cap \ZZ^{m+n}$.
\end{proof}

\begin{lem}\label{lem:canoG}
    We have $\mu(\omega_{\kk[G_{n,n,r}]})\ge r(n-3)+1$.
\end{lem}

\begin{proof}
    It is enough to show that the set of the lattice points in $\int(S_{G_{n,n,r}})$ corresponding to the minimal generators of $\omega_{\kk[G_{n,n,r}]}$ contains the following:
    \[
        v_{i,j} \coloneqq (1,\ldots,1,\overset{i}{\check{j}},1,\ldots,1,\overset{i+n}{\check{j}},1,\ldots,1) \text{ for }  i\in \{1,\ldots, n\}  \text{ and } j\in \{1,\ldots,n-2\}.
    \]
    We can see that $v_{i,j}$ belongs to $\int(S_{G_{n,n,r}})$ since $v_{i,j}\in H_k^>$ for any $k\in [2n]$ and $v_{i,j}\in H_{\set{l}}^>$ for any $l\in [r]$.
    If we can write $v_{i,j}=v'+\rho(e)$ for some $v'\in \int(S_{G_{n,n,r}})$ and $e\in E(G_{n,n,r})$, then $e=\set{i,i+n}$ since $v'\in H_k^>$ for any $k\in [2n]$, a contradiction to $\set{i,i+n}\notin E(G_{n,n,r})$.
    Therefore, $v_{i,j}$ cannot be written as a sum of an element in
    $\int(S_{G_{n,n,r}})$ and an element in $S_{G_{n,n,r}}\setminus \set{0}$, which is the desired result.
\end{proof}

\begin{thm}\label{thm:almGnr}
    The edge ring $\kk[G_{m,n,r}]$ is almost Gorenstein if and only if $m=n$.
\end{thm}

\begin{proof}
    If $\kk[G_{m,n,r}]$ is almost Gorenstein, then we have $h_{s(\kk[G_{m,n,r}])}=1$ from \Cref{thm:almh_s=1}, so we get $m=n$.

    Suppose that $m=n$.
    It follows from \Cref{prop:hilbG} and \Cref{lem:canoG} that
    \[
        r(n-3)=\widetilde{e}(\kk[G_{n,n,r}]) \ge \mu(\omega_R)-1 \ge r(n-3).
    \]
    Therefore, we have $\widetilde{e}(\kk[G_{n,n,r}])=\mu(\omega_R)-1$ and $\kk[G_{n,n,r}]$ is almost Gorenstein from \Cref{thm:alm}.
\end{proof}

\subsection{Observations and questions on almost Gorenstein edge rings}

\Cref{prop:hilbG} and \Cref{thm:almGnr} tell us that $\kk[G_{n,n,r}]$ is almost Gorenstein and its $h$-vector $(h_0,h_1,\ldots,h_s)$ satisfies the following condition:
\begin{equation}\label{h_condition}
    h_i=h_{s-i} \text{ for } i=0,2,3,\ldots,\floor{s/2}. \tag{\text{$*$}}
\end{equation}
As far as the authors know, the $h$-vectors of almost Gorenstein edge rings discovered so far satisfy the condition \Cref{h_condition}.
For example, the edge ring of a complete graph $K_{2m}$ is almost Gorenstein (\cite[Theorem~1.3]{higashitani2022levelness}), its $h$-vector has been computed (\cite[Theorem~3.10]{villarreal1996normality}) and satisfies condition \Cref{h_condition}.

Moreover, the edge ring of a complete multipartite graph $K_{1,n,n}$ ($n\ge 2$) is almost Gorenstein (\cite[Theorem~1.3]{higashitani2022levelness}), which is isomorphic to a certain Hibi ring (\cite[Proposition~2.2]{higashitani2022conic}).
The $h$-vector of this Hibi ring satisfies condition \Cref{h_condition} (\cite[Theorem~5.3]{higashitani2016almost}).

Furthermore, the $h$-vectors of the edge rings of a certain family of graphs $\calG_n$, consisting of $n$ triangles that share a single common vertex, have been investigated in \cite{higashitani2023h}.
According to \cite[Theorem~1.1]{higashitani2023h}, $\kk[\calG_n]$ is almost Gorenstein and its $h$-vector satisfies condition \Cref{h_condition}.

These results naturally pose us with the following question:
\begin{q}
    Do the $h$-vectors of almost Gorenstein edge rings satisfy condition \Cref{h_condition}?
\end{q}
Actually, the following example gives a negative answer to this question:

\begin{ex}
    Let $\calP$ be the Petersen graph (see \Cref{fig:Petersen graph}).
    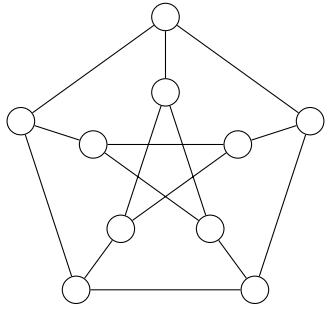
\begin{figure}[H]
        \centering
        \begin{tikzpicture}
            \foreach \i in {1,...,5} {
                    \node[draw, circle] (o\i) at (90+72*\i:2cm) {};
                    \node[draw, circle] (i\i) at (90+72*\i:1cm) {};
                    \draw (o\i) -- (i\i);
                }
            \draw (o1) -- (o2) -- (o3) -- (o4) -- (o5) -- (o1);
            \draw (i1) -- (i3) -- (i5) -- (i2) -- (i4) -- (i1);
        \end{tikzpicture}
        \caption{Petersen graph}\label{fig:Petersen graph}
    \end{figure}
    We can see that $\calP$ satisfies the odd cycle condition.
    Moreover, we can check
    \[
        h(\kk[\calP];t)=1+5t+15t^2+25t^3+5t^4+t^5
    \]
    by using \texttt{MAGMA} (\cite{magma}).
    It follows from \cite[Theorem~3.1]{higashitani2016almost} that $\kk[\calP]$ is not Gorenstein but almost Gorenstein.
    Moreover, this $h$-vector does not satisfy condition \Cref{h_condition}.
\end{ex}

We can still see that the $h$-vectors of almost
Gorenstein edge rings are ``almost symmetric'', which means that for $i=0,\ldots,\floor{s/2}$, the equality $h_i=h_{s-i}$ holds for all but at most one value of $i$.
Unfortunately, that is also not true in general.

\begin{ex}
    Let $W_{10}$ be the wheel graph with $10$ vertices (see \Cref{fig:wheel_10}).
    \begin{figure}[H]
        \scalebox{0.6}{
            \centering
            \begin{tikzpicture}
                \def\n{10}
                \foreach \i in {1,...,\n} {
                        \node[draw, circle, minimum size=12pt, inner sep=0pt] (P\i) at ({90+360/\n*\i}:3) {};
                    }
                \foreach \i in {1,...,\n} {
                        \pgfmathtruncatemacro{\nexti}{mod(\i,\n)+1}
                        \draw (P\i) -- (P\nexti);
                    }

                \node[draw, circle, minimum size=12pt, inner sep=0pt] (Center) at (0,0) {};
                \foreach \i in {1,...,\n} {
                        \draw (Center) -- (P\i);
                    }
            \end{tikzpicture}}
        \caption{The wheel graph $W_{10}$}\label{fig:wheel_10}
    \end{figure}
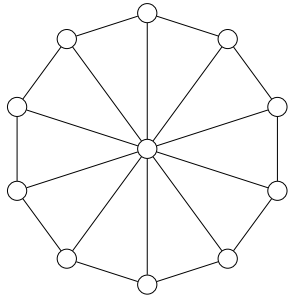
    This graph satisfies the odd cycle condition and
    \texttt{MAGMA} tells us that the $h$-polynomial of $\kk[W_{10}]$ is
    \[
        h(\kk[W_{10}];t)=1+8t+27t^2+30t^3+9t^4+t^5
    \]
    and that $\mu(\omega_{\kk[W_{10}]})=7$.
    Therefore, we have $\widetilde{e}(\kk[W_{10}])=\mu(\omega_{\kk[W_{10}]})-1=6$, and hence $\kk[W_{10}]$ is almost Gorenstein from \Cref{thm:alm}.
\end{ex}

While these counterexamples exist, we have yet to find the edge ring of a ``bipartite graph'' that is almost Gorenstein and whose $h$-vector does not satisfy condition \Cref{h_condition}.

\begin{q}
    Do the $h$-vectors of the almost Gorenstein edge rings of bipartite graphs satisfy condition \Cref{h_condition}?
\end{q}

\bibliographystyle{plain}
\bibliography{References}

\end{document}